\documentclass[12pt,final]{amsart}
\usepackage[mathscr]{euscript}
\usepackage{graphicx}
\usepackage{hyperref}
\usepackage{booktabs}
\usepackage{amsthm}
\usepackage{placeins}
\usepackage{float}

\title{Alternative links, homogeneous links, and graphs}
\author{Jeremy Siegert}

\newtheorem{theorem}{Theorem}
\newtheorem{lemma}[theorem]{Lemma}
\newtheorem{corollary}[theorem]{Corollary}

\newtheorem{definition}[theorem]{Definition}

\newtheorem{question}[theorem]{Question}

\begin{document}
\begin{abstract}
In this paper we review the definitions of homogeneous and alternative links. We also give two new characterizations of an alternative link diagram, one within the context of the enhanced checkerboard graph and another from the labeled Seifert graph.
\end{abstract}
\maketitle

\section{Introduction}
In \cite{LK} Kauffman introduced the class of alternative links and conjectured that the class of alternative links is identical to the class of pseudoalternating links introduced by Murasugi and Mayland in \cite{EM}. It was known that all alternative links were pseudoalternating, and it would be shown in \cite{MS} that all pseudoalternating links of genus one were alternative, but not necessarily in cases of higher genera, thus resolving Kauffman's conjecture in the negative. This was done by looking at the intermediary class of homogeneous links which were introduced in \cite{PC}. The remaining question is then:
\begin{question}
Are the classes of alternative and homogeneous links identical?
\end{question}
Studying these two classes is interesting because both classes of links have the interesting property of having a minimal genus Seifert surface that can be realized by Seifert's algorithm (\cite{LK},\cite{PC}). The difficulty in answering our question lies in the fact that the alternativity of a link is difficult to determine. In this paper we will review homogeneous and alternative links and then give two new characterizations of an alternative diagram. First, in sections 4 and 5 we explain how to construct the enhanced checkerboard graph and the labeled Seifert graph, respectively. In section 6 the definition of a homogeneous link is discussed. The definition of alternative links and our new characterizations of an alternative diagram are discussed in section 7. We conclude with some open questions in section 8. A basic knowledge of knot theory and graph theory is assumed throughout however some preliminary concepts are covered in sections 2 and 3. To avoid confusion, we will refer to a regular projection of an oriented link into the plane as a diagram.

\section{Preliminaries}
 The Seifert Algorithm is a procedure by which one can obtain a surface bounded by a link $L$ from one of its diagrams, \cite{LK}. Seifert's algorithm proceeds as follows. Given a diagram $D$ we smooth all of the crossings according to the following scheme.
\begin{figure}[H]
\centering

\includegraphics{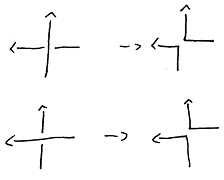}
\end{figure}
Following this procedure produces a collection of oriented circles in the plane bounded by discs. These circles are called Seifert circles. A diagram of $10_{138}$ before and after smoothing the crossings is displayed below.
\begin{figure}[H]
	\centering
	\caption{}
	\includegraphics{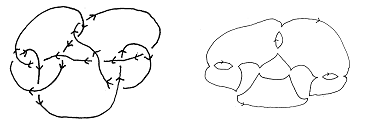}
\end{figure}

 To create a Seifert surface from these circles attach a twisted band to the discs bounding the circles in place of the crossings. The band is twisted in accordance with the smoothed crossing between the two circles. If one circle should be properly contained in another lift the innermost circle from the plane so that it lies above the plane containing the circle that properly contained it. If instead of creating a surface we place the appropriate positive or negative site markers as in \cite{LK} we acquire a diagram of Seifert circles with site markers between them instead of crossings. These site markers are shown below.
\begin{figure}[H]
\centering
\caption{Top: positive marker Bottom: negative marker}
\includegraphics{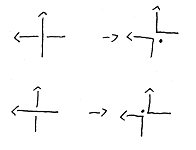}
\end{figure}

The components of the complement of the union of Seifert circles are referred to in \cite{LK} as \emph{spaces}. We will refer to the union of circles created by replacing the crossings of a diagram $D$ with site markers as the Seifert diagram, denoted $D_{s}$. Similarly we will denote the collection of spaces in the complement of the union of Seifert circles as $\hat{D}_{s}$. Pictured below is the Seifert diagram of the diagram previously shown.
\begin{figure}[H]

\centering
	\caption{$D_{s}$}
	\includegraphics{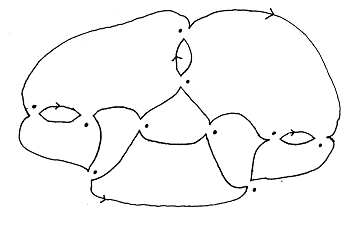}
	
\end{figure}
In this paper we will often refer to the \emph{height} of a Seifert circle as defined in \cite{PM}:
\begin{definition}
Given a Seifert circle $C$ of a Seifert diagram $D_{s}$ the height of $C$, denoted $h(C)$ is defined as the number of Seifert circles in $D_{s}$ that properly contain $C$.
\end{definition}
We will also, on occasion refer to the height of a space bounded by a Seifert circle. In this case we are referring to the bounded space in the plane properly contained between Seifert circles or the large unbounded space surrounding the diagram, more formally:
\begin{definition}
Given a Seifert diagram $D_{s}$ a space $A\subseteq \hat{D}_{s}$ is said to be of height $k$ if and only if it is bounded by a single Seifert circle of height $k$ or by a Seifert circle of height $k$ and one or more Seifert circles of height $k+1$. A space of the Seifert diagram is said to be of height $-1$ if and only if it is unbounded.
\end{definition}

A diagram divides the plane into several regions, including the large outside region surrounding the diagram. We will say that two regions $R_{1}$ and $R_{2}$ of a diagram are touching if they meet at a crossing and we will say that the crossing is incident to both $R_{1}$ and $R_{2}$. For example, the regions $A$, $B$, $C$ and $D$ in the figure below are touching. 
\begin{figure}[H]
  \centering
  
    \includegraphics{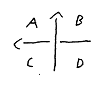}
\end{figure}
Observe that in the above diagram two arcs begin bordering $A$ and stop bordering $D$ at the crossing according to the orientation. Similarly one arc begins bordering $B$ and $C$ and stops bordering $B$ and $C$. We will say that if an arc begins bordering a region $R$ at a crossing $c$ then that arc is \emph{going into} $R$. Similarly, if an arc ceases to border a region $R$ at a crossing $c$ then the arc is said to be \emph{going out} of $R$. We then define the following:
\begin{definition}
Let $c$ be a crossing incident to a region $R$ then $c$ is:\\
\mbox{ }\\
a. Crossing into $R$ if and only if two arcs begin bordering $R$ according to the orientation at $c$.\\
\mbox{ }\\
b. Crossing out of $R$ if and only if two arcs stop bordering $R$ according to the orientation at $c$.\\
\mbox{ }\\
c. Sideswiping $R$ if and only if one arc begins bordering $R$ and one arc stops bordering $R$ according to the orientation at $c$.
\end{definition}
In the previous diagram the crossing show is crossing into $A$, out of $D$ and sideswiping $B$ and $C$.
 \begin{lemma}
 The number of arcs going into a region $R$ is equal to the number of arcs leaving $R$.
 \end{lemma}
 
 \begin{proof}
 Every arc into a region $R$ must connect to a unique arc going out of $R$. If not then we could have a single arc into $R$ attaching to multiple arcs leaving $R$ which is impossible or we could have two arcs into $R$ connecting to a single arc going out of $R$ which is also impossible. It then must be that every arc going into a region $R$ must connect to one and only on arc going out of $R$ and similarly every arc going out of $R$ must have come from one and only one arc coming into $R$. We then have that the number of arcs going into $R$ is equal to the number going out of $R$.
 \end{proof}
Note that if $R_{1}$ and $R_{2}$ are two regions of a link diagram with a crossing $c$ that is going out of $R_{1}$ and into $R_{2}$ then, upon smoothing the crossings of the diagram according to the Seifert algorithm $R_{1}$ and $R_{2}$ will be joined in the same space of the Seifert diagram. Similarly if there is a crossing out of $R_{2}$ into a region distinct from $R_{1}$, say $R_{3}$, then upon smoothing the crossings $R_{1}$,$R_{2}$, and $R_{3}$ will also all be joined in the same space of the Seifert diagram. We then introduce the following relation:
\begin{definition}
Given a diagram $D$ with regions $R_{0},R_{1},\dots, R_{n}$ we will say that two regions $R$ and $R^{\prime}$ are spatially connected, denoted $R\sim R^{\prime}$, if and only if there is a sequence of regions $R_{1},\dots,R_{n}$ with $R=R_{1}$ and $R^{\prime}=R_{n}$ such that for every pair of regions $R_{i}$ and $R_{i+1}$ there is a crossing $c$ that is either crossing out of $R_{i}$ and into $R_{i+1}$ or otherwise crossing into $R_{i}$ and out of $R_{i+1}$.
\end{definition}

\section{Graph Theory}

A basic knowledge of Graph Theory is assumed throughout this paper, however some terms are used frequently enough to merit a restatement of their definitions as taken from \cite{JA}:

\begin{definition}
Given a graph $G$ with vertex set $V$ a vertex $v\in V$ is a cut vertex if $G\setminus v$ is has more connected components than $G$. 
\end{definition}

\begin{definition}
A separation of a connected graph is a decomposition of the graph into two nonempty connected subgraphs which have just one vertex in common. This common vertex is called a separating vertex of the graph. A graph is said to be separable if it contains a separating vertex.
\end{definition}
Pictured below is a graph and it's smallest nontrivial nonseparable subgraphs:
\begin{figure}[H]
\centering
\includegraphics{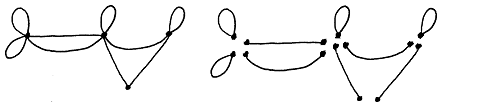}
\end{figure}

\begin{definition}
A block $B$ of a finite graph $G$ is a subgraph which is nonseparable and is maximal with respect to this property.
\end{definition}
An example of a graph and it's blocks is displayed below:
\begin{figure}[H]
	\centering
	\includegraphics{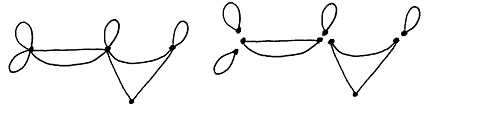}
\end{figure}
Given a digraph $G$ we denote the in degree of a vertex $v$ as $\delta^{-}(v)$, the outdegree of $v$ as $\delta^{+}(v)$ and the degree of $v$ as $d(v)=\delta^{-}(v)+\delta^{+}(v)$.
We will also the use the following theorem from section 2.4 of \cite{JA}:
\begin{theorem}
A graph $G$ has a cycle decomposition if and only if it is an even graph.
\end{theorem}

\section{The Enhanced Checkerboard Graph}
Let $D$ be a diagram. Checkerboard color the diagram as in \cite{LK}. We can construct a signed planar digraph $\Phi(D)$ called the enhanced checkerboard graph from this diagram in the following way. Assign a black vertex to each black region and a white vertex to each white region including the large ``outside" region of the diagram. Two vertices are connected by a directed edge labeled either $+$ or $-$ according to the scheme shown below. A clarification of how loops are handled is also shown.

\begin{figure}[H]\label{crossingscheme1}
  \centering
  
    \includegraphics{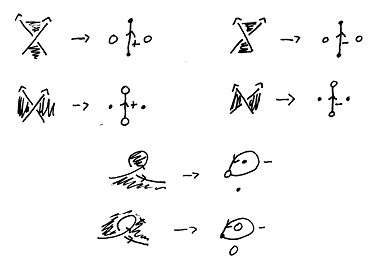}
\end{figure}
To more clearly illustrate how $\Phi(D)$ is constructed a colored diagram and the corresponding $\Phi(D)$ of $10_{138}$ are shown below. Note that, by construction $\Phi(D)$ is necessarily planar. This is due to the fact that our diagrams come from regular projections which only admit single and double points.
\begin{figure}[H]\label{loops}

     \includegraphics{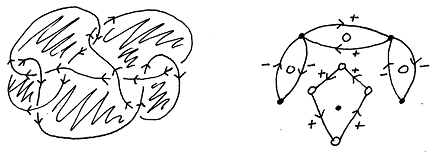}
 \end{figure}

 A similar graph construction is explained in \cite{PM} from which we borrow the $\Phi$ notation.
 
 \begin{theorem}
 For every vertex $v\in \Phi(D)$, $\delta^{+}(v)=\delta^{-}(v)$.
 \end{theorem}
 \begin{proof}
 Let $v$ be an arbitrary vertex of $\Phi(D)$. Assume for the sake of contradiction that $\delta^{-}(v)\neq\delta^{+}(v)$ We then have that the number of crossings into the region corresponding to $v$, say $R_{v}$, is not equal to the number of crossings out of $R_{v}$. This would imply that the number of arcs into $R_{v}$ is not equal to the number of arcs leaving $R_{v}$ because the only other crossings with arcs going into $R_{v}$ are sideswiping which contribute one arc into $R_{v}$ and one arc out. We know that the number of arcs going into $R_{v}$ must equal those leaving so we have a contradiction from which we can conclude that $\delta^{-}(v)=\delta^{+}(v)$.
 \end{proof}
  In other words, $\Phi(D)$ is an even graph and has a cycle decomposition that is to say every edge is a part of a cycle.
\begin{theorem}
Two regions $R_{1}$ and $R_{n}$ of a diagram $D$ are part of the same space in the Seifert diagram $\hat{D}_{s}$ if and only if their is a (not necessarily directed) path between their corresponding vertices in $\Phi(D)$.
\end{theorem}
\begin{proof}
($\Rightarrow$) Let $R_{1}$ and $R_{n}$ be two regions that are part of the same space in the Seifert diagram. Then $R_{1}$ and $R_{2}$ are spatially connected which means there is a sequence of regions $R_{1},\dots,R_{n}$ such that for $1\leq i\leq n$ there is a crossing $c$ in $D$ that is crossing out of $R_{i}$ and into $R_{i+1}$ or otherwise into $R_{i}$ and out of $R_{i+1}$. In either case, the vertices corresponding to $R_{i}$ and $R_{i+1}$ are joined by an edge in $\Phi(D)$. We may then conclude that there will be a path in $\Phi(D)$ from the vertex corresponding to $R_{1}$ to that of $R_{n}$.\\
($\Leftarrow$) Let $v_{0}$ and $v_{n}$ be distinct vertices of $\Phi(D)$ such that there is a (not necessarily directed path) path $P=v_{0}, v_{1},v_{2},\dots,v_{n}$ between them. Then we have that there is a crossing $c$ in $D$ that is either crossing out of $R_{i}$ and into $R_{i+1}$ or vice versa for $0\leq i\leq n$ which gives us that $R_{0}$ and $R_{n}$ are spatially connected and are therefore in the same space of $\hat{D}_{s}$. This completes the proof.

\end{proof}
We can then view a component of $\Phi(D)$ as a characterization of the site markers in a particular space of a given $\hat{D}_{s}$.

\section{The Seifert Graph and the Labelled Seifert Graph}
 As shown in \cite{PC} we can construct the Seifert graph of a diagram from it's corresponding Seifert diagram. This is done by assigning a vertex to each Seifert circle and a signed edge labeled either $+$ or $-$ between two vertices for every site marker between their corresponding Seifert circles. The sign of these edges is determined by the marker it corresponds to. Edges with the same label correspond to markers of the same type. More precisely, a negative site marker corresponds to a negatively signed edge and positive site markers a positively signed edge. For clarity we will represent positive edges as solid and negative edges as hashed. We will say that two edges are of the same type if they are both solid or are both hashed. An example of a Seifert graph from the same diagram of $10_{138}$ that we have used in previous sections is displayed below.
 \begin{figure}[H]
 \centering
 \caption{}
 \includegraphics{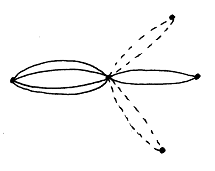}
 
 \end{figure}
 Here we define an extension of the Seifert graph that we call the labeled Seifert graph defined in the following way, 
 \begin{definition}
 The class of labeled Seifert graphs $\mathscr{LS}$ is the class containing the graphs $G$ that meet the following conditions:\\
 \begin{tabular}{ll}
 i. & The vertices of $G$ are labeled with nonnegative integers.\\
 ii. & The edges of $G$ are labeled $+$ or $-$.
 \end{tabular}
 \end{definition}
 When we consider a labeled Seifert graph constructed from a particular diagram $D$ we denote that graph as $G(D)$. $G(D)$ is constructed in the same way as the original Seifert graph with the only addition being that the vertices are labeled with the height of the Seifert circle that they correspond to. This is to say if $C$ is a Seifert circle with $h(C)=i$ then the corresponding vertex in $G(D)$ is labeled with an $i$. Displayed below is the same Seifert graph as before only now as a labeled Seifert graph:
 \begin{figure}[H]
  \centering
  \caption{}
  \includegraphics{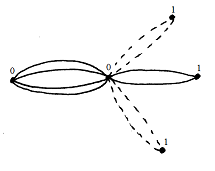}
  
  \end{figure}

Consider an arbitrary labeled Seifert graph $G$. Let $G_{i}$ be the subgraph of $G$ constructed in the following way. Delete all vertices $v$ such that $h(v)< i-1$ or $h(v)>i$. Then delete all edges between vertices of height $i-1$. This potentially disconnected subgraph is $G_{i}$. We will refer to all $G_{i}$'s as \emph{height subgraphs}. Observe that $G=\bigcup_{i=0}^{n}G_{i}$ where $n$ is the maximum height of a circle in the Seifert diagram. For clarity, when discussing a labeled Seifert graph constructed from a diagram $D$ we will denote the height subgraphs as $G_{i}(D)$. For example, in the figure below are the $G_{i}(D)$'s for the same projection of $10_{138}$ we have been using.

\begin{figure}[H]
  \centering
  \caption{$G_{0}(D)$ and $G_{1}(D)$}
    \includegraphics{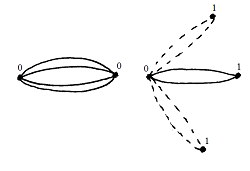}
\end{figure}

\section{Homogeneous Links}
Homogeneous Links were defined in \cite{PC} in the following way.

\begin{definition}
Given a diagram $D$ and Seifert graph $G$ we say that a block $B_{i}$ of $G$ is homogeneous if all edges in $B_{i}$ are of the same type.  We say that $D$ is homogeneous if all blocks in $G$ are homogeneous. We say that $L$ is homogeneous if it has a homogeneous diagram.
\end{definition}
Note the following result from \cite{PM}:
\begin{theorem}\label{block}
Let $D$ be a diagram, $F$ it's corresponding Seifert surface and $G$ the corresponding Seifert Graph. Then all the Seifert circles associated to a block of $G$ have the same height, except possibly one of them which contains all the other being its height one less.
\end{theorem}

Given that our labelled Seifert graph is an extension of the Seifert graph this result will hold for labelled Seifert graphs as well. From this we may conclude the following:
\begin{corollary}
Let $B$ be a block of some labelled Seifert graph $G(D)$, then there is some $i$ such that $B\subset G_{i}(D)$.
\end{corollary}

\section{Alternative Links}

Recall that in the construction of the Seifert diagram crossings are replaced with positive or negative site markers. Two markers are said to be of the \emph{same type} if they are both positive or are both negative. A space in the Seifert diagram may contain one, two or zero marker types. We then have the following definition from \cite{LK}:
\begin{definition}
 A diagram $D$ is alternative if and only if each space in it's corresponding Seifert diagram $D_{s}$ contains at most one marker type. $L$ is said to be alternative if it has an alternative diagram.
 \end{definition}

We can recharacterize this definition in the context of the spatial graph of a diagram, but first we will need the following definition.
\begin{definition}
Given a spatial graph $\Phi(D)$ we will say that $\Phi(D)$ is alternative if and only if it does not contain a walk with both positive and negative edges.
\end{definition}

\begin{theorem}
Let $D$ be a diagram, then $D$ is alternative if and only if $\Phi(D)$ is alternative.
\end{theorem}
\begin{proof}
($\Rightarrow$) Let $D$ be an alternative diagram. Consider $\Phi(D)$. Each crossing of $D$ corresponds to an edge in $\Phi(D)$. Recall that two edges in a connected component of $\Phi(D)$ correspond to markers in the same space of the Seifert diagram of $D$. Because $D$ is alternative we know that each space contains only one marker type we then have that each connected component of $\Phi(D)$ has only one edge type which is to say there is no walk in $\Phi(D)$ which contains both negative and positive edges.\\
($\Leftarrow$) Let $\Phi(D)$ be the spatial graph of the oriented link diagram $D$. Assume also that $\Phi(D)$ is alternative. Then there is no walk in $\Phi(D)$ which contains both negative and positive crossings. As we have noted, the edges of a connected component of $\Phi(D)$ corresponds to the edges in a space of the Seifert diagram of $D$. Because there is no walk containing both edge types in $\Phi(D)$ we  then know that there is no space in the Seifert diagram of $D$ that has more than one edge type which is to say that $D$ is alternative. This completes the proof.
\end{proof}

We can also recharacterize the definition of an alternative diagram in terms of the labeled Seifert graph.
\begin{definition} Given a labeled Seifert graph $G$ we will say that a subgraph $G_{i}$ is alternative if and only if each connected component of $G_{i}$ has at most one edge type. The graph $G$ is alternative if and only if each $G_{i}$ is alternative.
\end{definition}
We may then present our main theorem, a recharacterization of an alternative diagram in the context of labeled Seifert graphs.

\begin{theorem}
Let $D$ be a diagram and $G(D)$ the labeled Seifert graph of $D$, then $D$ is alternative if and only if $G(D)$ is alternative.
\end{theorem}
\begin{proof}
($\Rightarrow$). Let $D$ be an alternative diagram. Consider $G(D)$. Assume for the sake of contradiction that $G(D)$ is not alternative. Then there is some $i$ such that $G_{i}(D)$ is not alternative. This is to say there is some connected component in $G_{i}(D)$ that has two edges of different types. Let this component be denoted $H$. Let $v_{0}$ be the single vertex in $H$ that is at height $i$ and let all other vertices be denoted by $w_{1},\dots ,w_{n}$. There are then three cases to check. In the first case there are two edges incident to $v_{0}$ that are of different types. In this case it is easy to see that these edges correspond to crossings of different types in the space corresponding to $v_{0}$ which would give that $D$ is not alternative. In the second case two edges between $w_{j}$ , $w_{k}$, and $w_{l}$ are of different types. All edges between these $w_{i}$'s correspond to crossings in the space associated to $v_{0}$. If any two of these are of different types then we have that $D$ is not alternative.  In the final case there is an edge incident to $v_{0}$ that is a different type than an edge between two $w_{i}$'s. We know that edges incident to $v_{0}$ in $G_{i}(D)$ correspond to crossings in the space associated to $v_{0}$ and similarly, edges amongst the $w_{i}$'s correspond to crossing in the space corresponding to $v_{0}$. We then have that if there are differing types between these edges then $D$ must not be alternative. This is our final contradiction. We then must have that all $G_{i}(D)$ are alternative thereby making $G(D)$ alternative by definition.\\
($\Leftarrow$) Let $G(D)$ be an alternative labeled Seifert graph for some diagram $D$. We wish to show that there is no space in the Seifert diagram of $D$ that has more than one type of site marker. Assume for the sake of contradiction that there is such a space. If this space is the  larger outside space then this would imply that there is a Seifert circle of height 0 that is connected to other distinct circles of height 0 by different site markers. This would imply that there are edges of differing types in the subgraph $G_{0}(D)$ which is a contradiction. Then consider a space between circles of height $k$ and height $k+1$. If this space has differing site markers then the circle of height $k$ must contain circles of height $k+1$. If these circles have crossings of different types amongst them then these crossings will show as different types of edges in $G_{k}(D)$. Similarly for the other 2 possible cases. These would all contradict $G(D)$ being alternative, so it must be that $D$ is alternative.
\end{proof}
It was noted in \cite{PC} that all alternative link diagrams are homogeneous. This is easy to see because if a labeled Seifert graph is alternative then all of its $G_{i}(D)$'s are alternative which would imply that the blocks contained in each $G_{i}(D)$ are homogeneous. However, not all homogeneous diagrams are alternative. Below are two diagrams of a $9_{43}$ with their respective labeled Seifert graphs. The second diagram is alternative and homogeneous while the first diagram is homogeneous, but not alternative.
 \begin{figure}[H]
   
   \caption{A homogeneous diagram that is not alternative}
   
    \includegraphics{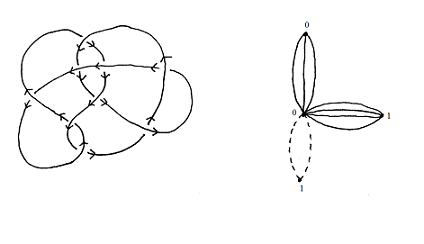}
 \end{figure}
 
 \begin{figure}[H]
 \caption{A diagram that is both homogeneous and alternative}
 
 \includegraphics{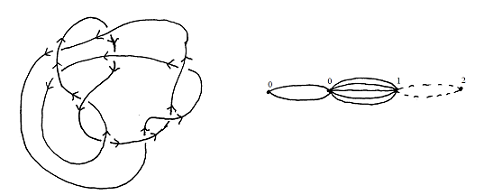}
 \end{figure}
Similarly, it is not the case that all minimal crossing diagrams of a homogeneous link are homogeneous as can be seen in the infamous ``Perko Pair" which answers question 1 from \cite{PC}. This example also shows that minimal crossing diagrams of alternative links are not necessarily alternative. Note that $10_{161}$ might not the smallest homogeneous link that has a nonhomogeneous minimal crossing diagram. All minimal crossing diagrams of alternating links are alternating and therefore homogeneous and alternating. It must then be that such a minimal example is nonalternating. The smallest nonalternating homogeneous link is $8_{19}$, \cite{PC}. The sequence of Reidemeister moves between the two diagrams of $10_{161}$ with the corresponding spatial and labeled Seifert graphs are included in the appendix.
\begin{figure}[H]
   
   \caption{A minimal crossing homogeneous diagram and a minimal crossing nonhomogeneous diagram of $10_{161}$}
   
     \includegraphics{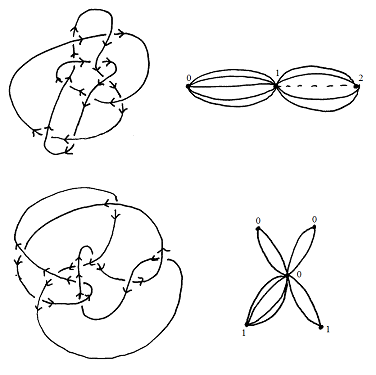}
 \end{figure}

\section{Questions}
In \cite{PC} the question was asked as to whether or not every homogeneous link has a homogeneous minimal crossing diagram. It is then natural to ask the same of alternative links.
\begin{question}
 Does every alternative link have an alternative minimal crossing diagram?
\end{question}
\begin{question} 
Given a labeled Seifert graph $G(D)$ of a diagram $D$ is there a sequence of Reidemeister moves that can be done on $D$ to get a diagram $D^{\prime}$ where the labeled Seifert graph, $G(D^{\prime})$ has different heights, but is isomorphic to $G(D)$ otherwise? 
\end{question}
\begin{figure}[H]
   
   \caption{$G(D)$ and $G(D^{\prime})$}
   
     \includegraphics{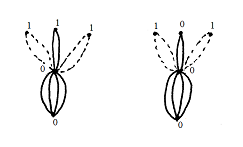}
 \end{figure}
\begin{question}
Is there a homogeneous link that is not alternative?
\end{question}
There is a complete classification of homogeneous knots up to 9 crossings in \cite{PC}. All homogeneous links up to $10_{134}$ can be seen to be alternative from their minimal crossing diagrams given on \cite{JC}. The first homogeneous knot listed that is not immediately identifiable as alternative is $10_{138}$.
\begin{figure}[H]
   
   \caption{Does $10_{138}$ admit an alternative diagram?}
   
     \includegraphics{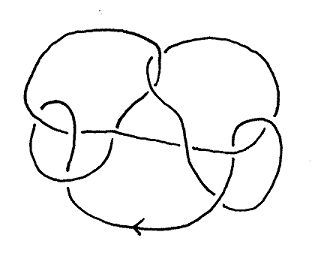}
 \end{figure}

\section{Appendix}
Presented here is the Reidemeister sequence between the two diagrams of the Perko Pair. At each step the knot diagram is on the left and the corresponding spatial and labeled Seifert graph are on the center and right respectively.\\

\includegraphics{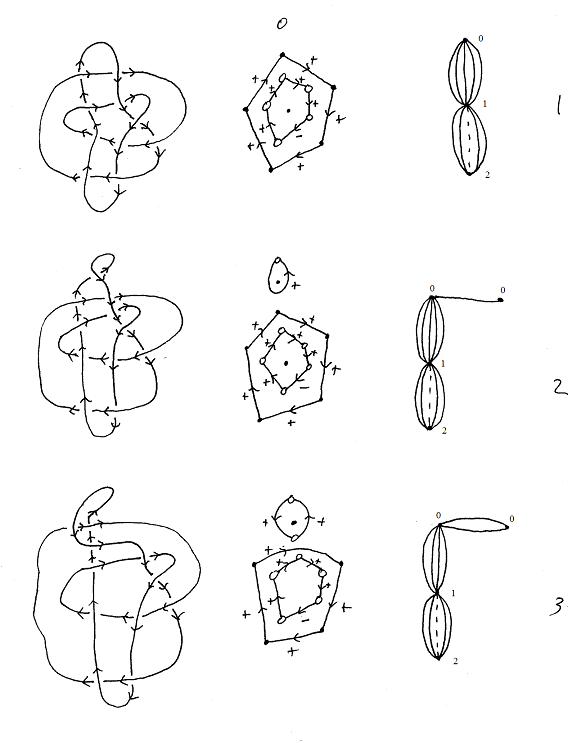}
\newpage
\includegraphics{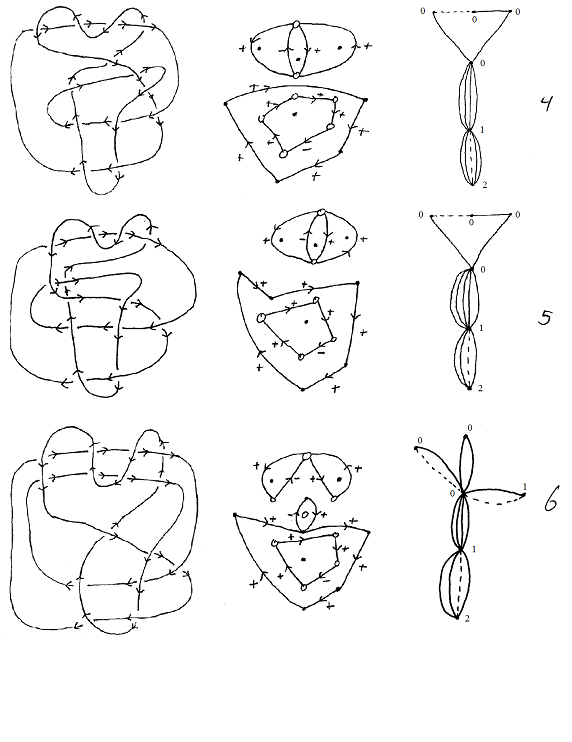}
\newpage
\includegraphics{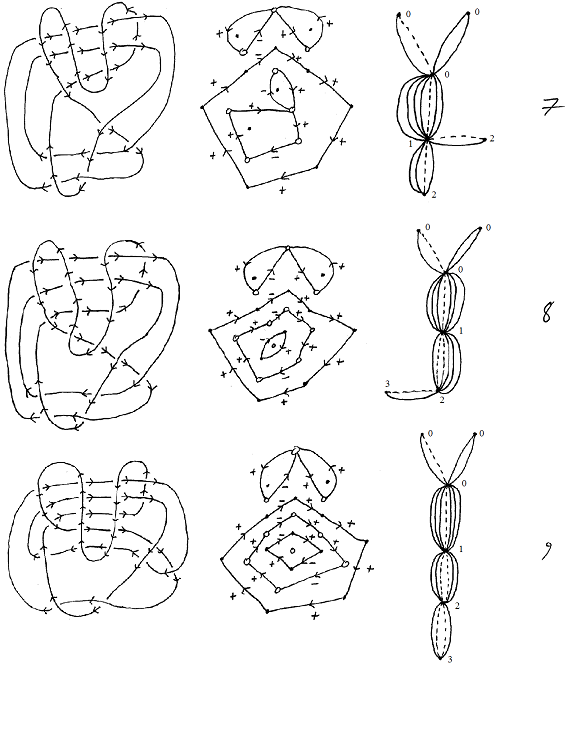}
\newpage
\includegraphics{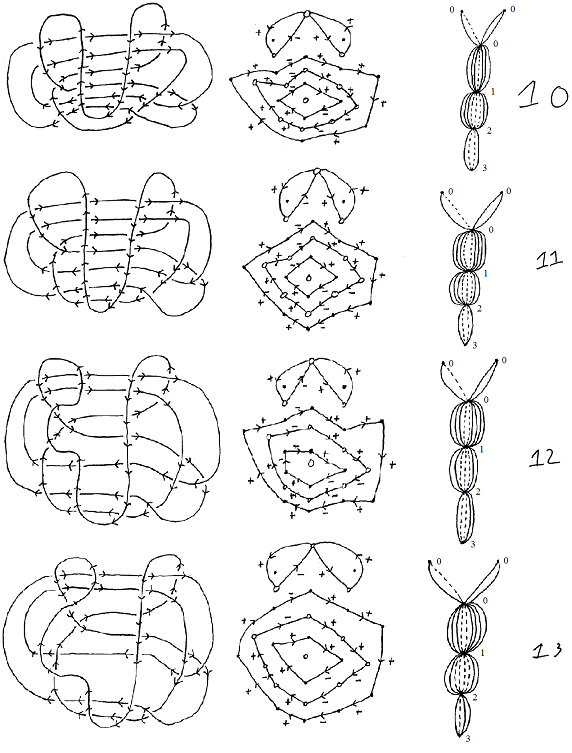}
\newpage
\includegraphics{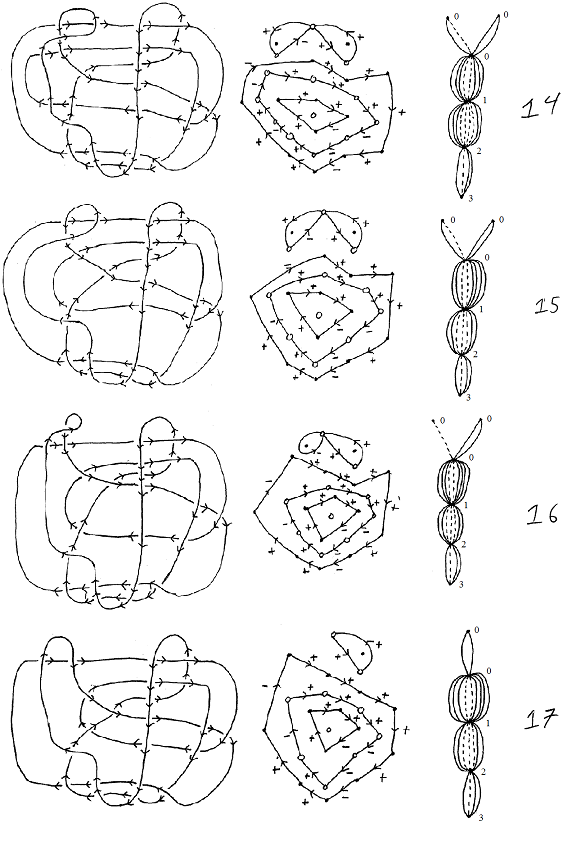}
\newpage
\includegraphics{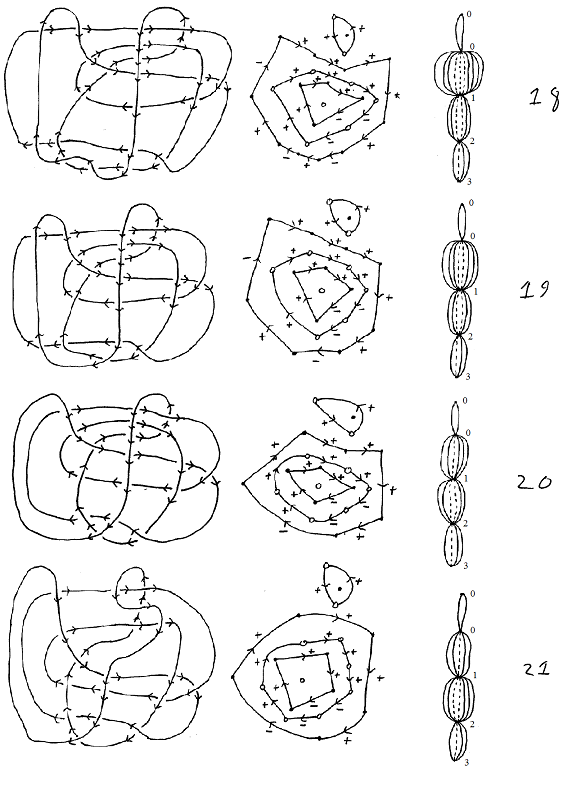}
\newpage
\includegraphics{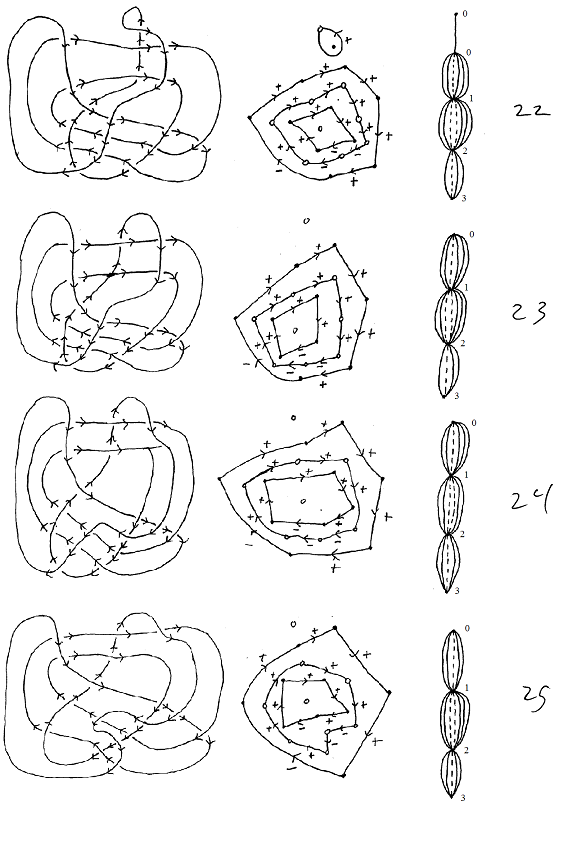}
\newpage
\includegraphics{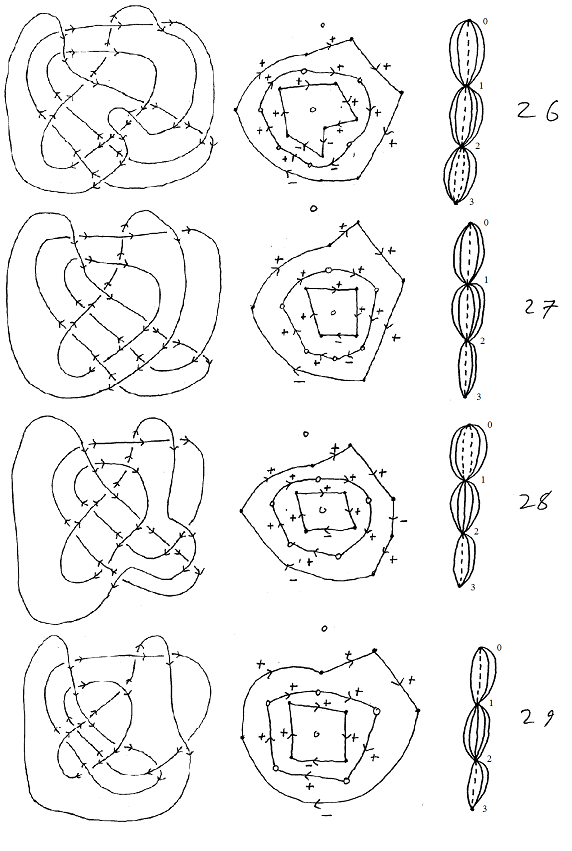}
\newpage
\includegraphics{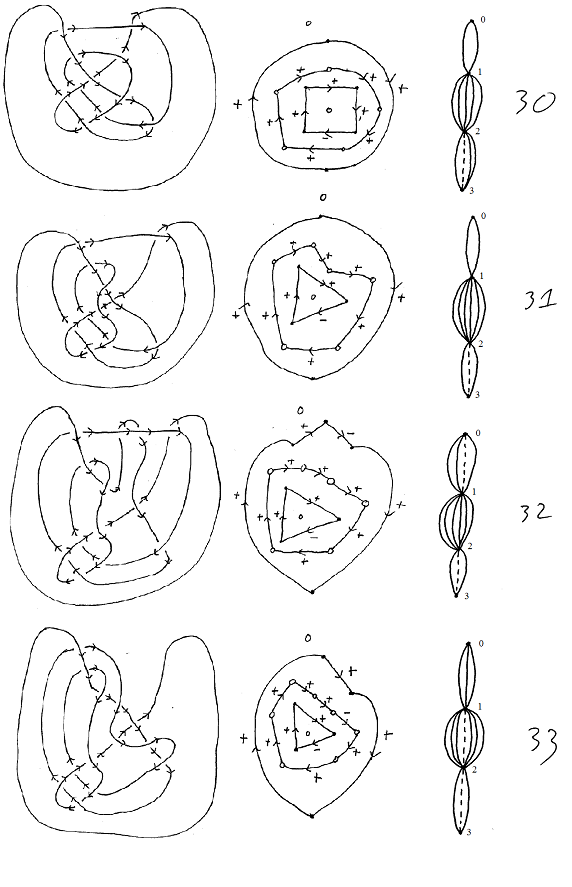}
\newpage
\includegraphics{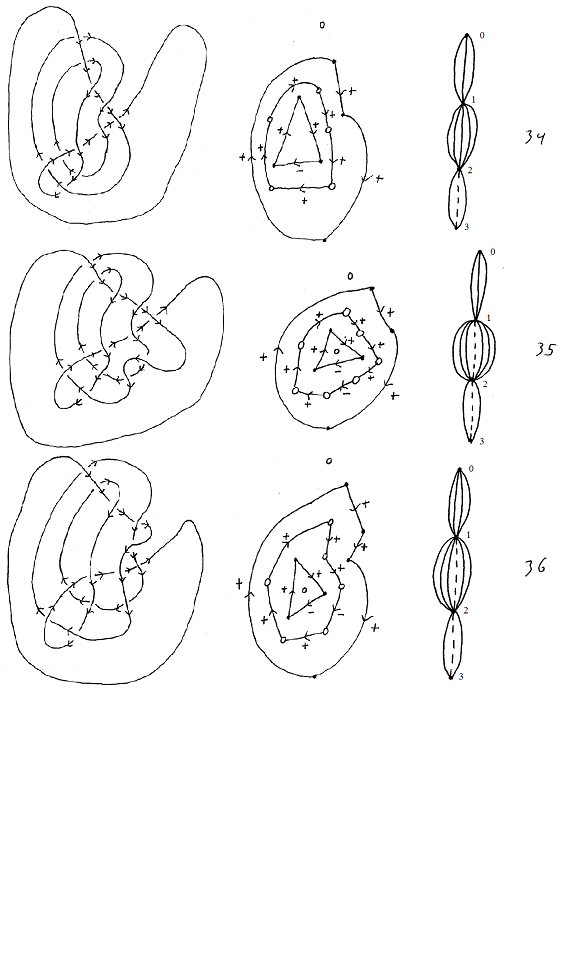}
\newpage
\includegraphics{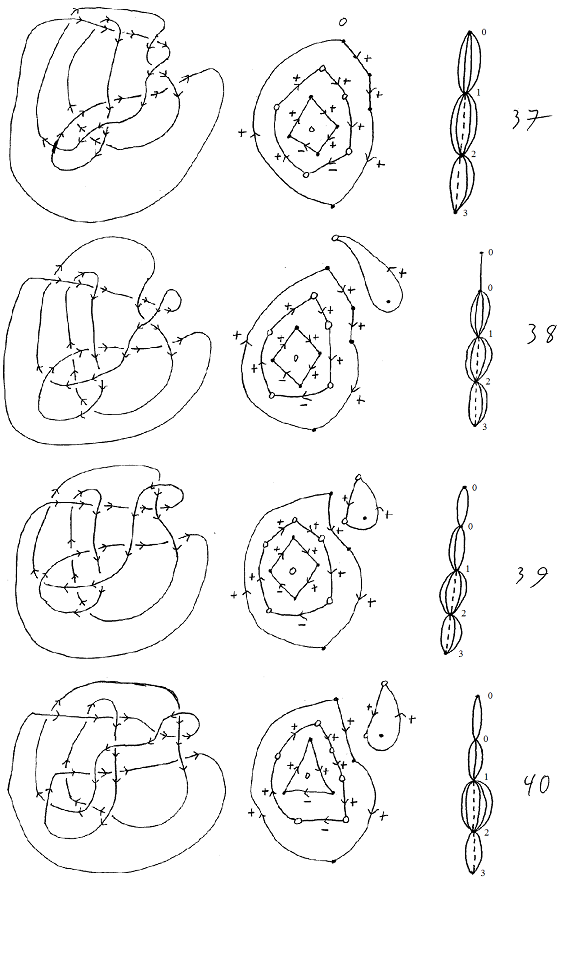}
\newpage
\includegraphics{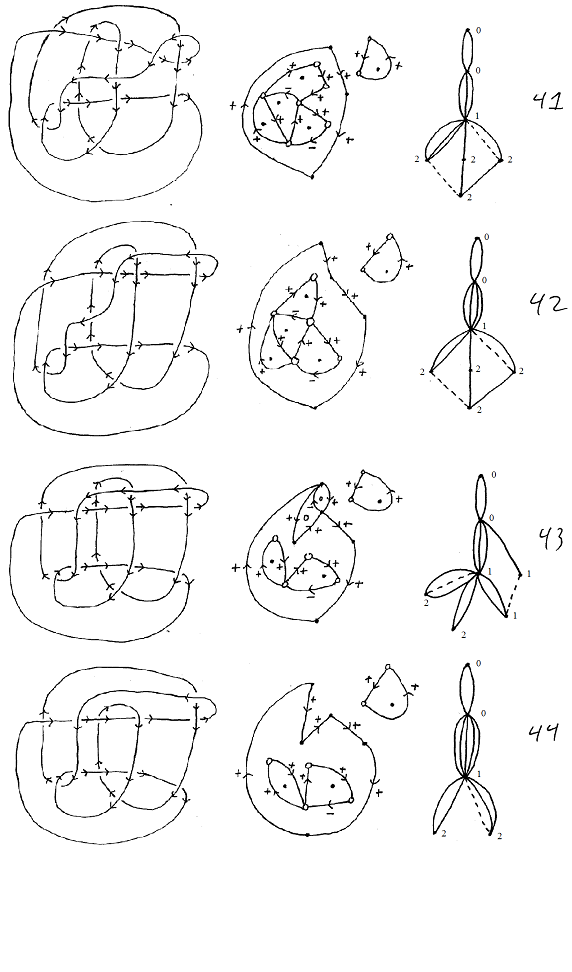}
\newpage
\includegraphics{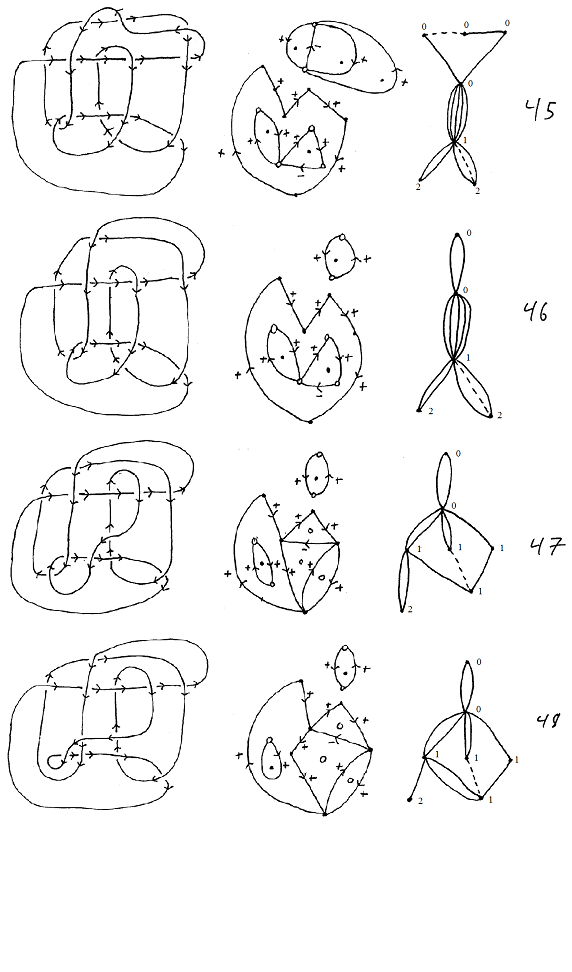}
\newpage
\includegraphics{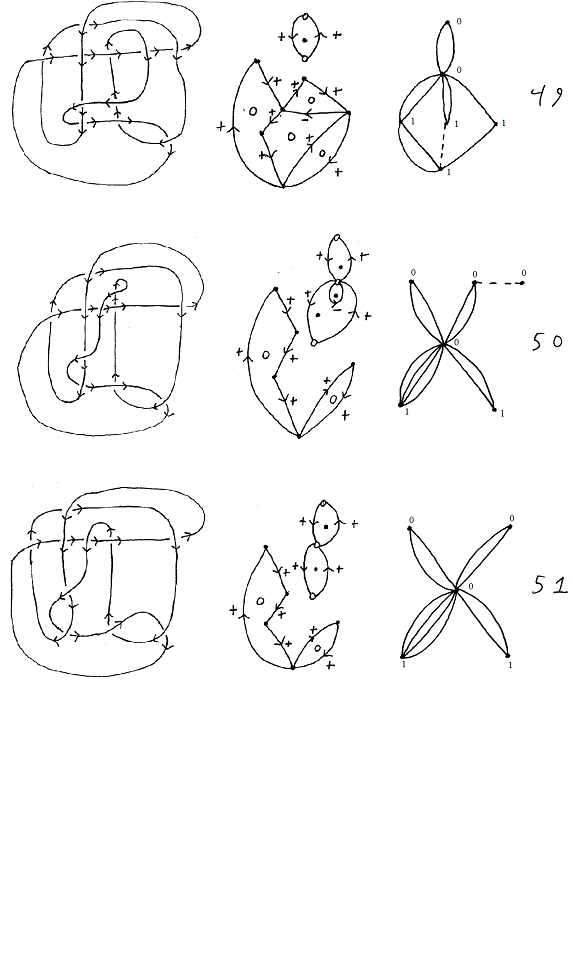}
\end{document}